\documentclass[12pt,reqno,a4paper]{article}

\usepackage[normalem]{ulem}

\usepackage{lipsum}
\usepackage{fullpage}
\usepackage{setspace}
\usepackage{amssymb}
\usepackage{amsmath}
\usepackage{amsthm}
\usepackage{subcaption}
\usepackage[utf8]{inputenc}
\usepackage[T1]{fontenc}
\usepackage{ae}
\usepackage{aecompl}
\usepackage{verbatim}
\usepackage[normalem]{ulem}
\usepackage[usenames,dvipsnames]{color}
\usepackage{graphicx}
\usepackage{tikz}
\usetikzlibrary{arrows}
\usetikzlibrary{calc}
\usetikzlibrary{external}
\usepackage{array}
\usepackage{bm}
\usepackage{stmaryrd}
\usepackage{enumitem}
\usepackage{mathtools}
\usepackage{upgreek}

\newtheoremstyle{plain}
  {\medskipamount}
  {\smallskipamount}
  {\slshape}
  {0pt}
  {\bfseries}
  {.}
  { }
  {\thmname{#1}\thmnumber{ #2}{\normalfont\thmnote{ (#3)}}}

\newtheorem{theorem}             {Theorem}[section]
\newtheorem{lemma}     [theorem] {Lemma}

\newtheorem{definition}[theorem] {Definition}   
   
\newtheorem{corollary}[theorem]  {Corollary}

\newtheorem{claim}[theorem]      {Claim}

\newtheorem*{theorem*}{Theorem}
\newtheorem*{lemma*}{Lemma}

\def\Pr{\mathop{\text{\rm Pr}}\nolimits}
\def\var{\mathop{\text{\rm Var}}\nolimits}

\let\:=\colon

\def\bin{\mathop{\text{\rm Bin}}\nolimits}

\def\bbe{{\mathbb E}}

\def\bbn{{\mathbb N}}

\def\cald{{\mathcal D}}

\def\calf{{\mathcal F}}
\def\calg{{\mathcal G}}

\def\calp{{\mathcal P}}

\let\epsilon\varepsilon
\newcommand{\deltastar}{\delta^*} 
\newcommand{\eps}{\varepsilon}

\def\Pr{\mathop{\text{\rm Pr}}\nolimits}
\def\var{\mathop{\text{\rm Var}}\nolimits}

\usepackage[%
  pdftex,%
  plainpages=false,%
  pdfpagelabels,%
  colorlinks=false,%
  bookmarksopen,%
  hidelinks%
]{hyperref}
\usepackage[all]{hypcap}
\newcommand{\set}[2][]{#1\{#2#1\}}
\newcommand{\paren}[2][]{#1(#2#1)}

\newcommand{\floor}[2][]{#1\lfloor#2#1\rfloor}

\newcommand{\aas}{a.a.s.}

\newcommand{\mean}[1]{\mathbb{E}\left[#1\right]}

\newcommand{\deltain}{\delta^{in}}
\newcommand{\deltaout}{\delta^{out}}
\newcommand{\Dnp}{\cald(n,p)}
\newcommand{\din}{d^{in}}
\newcommand{\dout}{d^{out}}
\title{Packing arborescences in random digraphs}

\begin{document}

\author{%
  Carlos Hoppen%
  \thanks{Instituto de Matem\'{a}tica e Estat\'{i}stica, Universidade Federal do Rio Grande
    do Sul. Partially supported by CNPq (Proc.~448754/2014-2 and~308539/2015-0), FAPESP (Proc.~2013/03447-6) and NUMEC/USP (Project MaCLinC/USP).}%
  \and%
  Roberto F.~Parente%
  \thanks{Departamento de Ci\^encia da Computação, Universidade Federal da Bahia. Supported by CNPq (140987/2012-6) 
  while a doctoral student at Universidade de São Paulo.} \and%
  Cristiane M.~Sato%
  \thanks{Centro de Matemática, Computação e Cognição, Universidade
    Federal do ABC. Partially supported by FAPESP
    (Proc.2103/03447-6).}%
}

\date{\today}


\pagestyle{plain}

\thispagestyle{empty}

\maketitle
\begin{abstract}
We study the problem of packing arborescences in the random digraph $\mathcal D(n,p)$, where each possible arc is included uniformly at random with probability $p=p(n)$. Let $\lambda(\mathcal D(n,p))$ denote the largest integer $\lambda\geq 0$ such that, for all $0\leq \ell\leq \lambda$, we have $\sum_{i=0}^{\ell-1} (\ell-i)|\{v: d^{in}(v) = i\}| \leq \ell$. We show that the maximum number of arc-disjoint arborescences in $\mathcal D(n,p)$ is $\lambda(\mathcal D(n,p))$ \aas{} We also give tight estimates for $\lambda(\mathcal D(n,p))$ depending on the range of $p$.
\end{abstract}

\section{Introduction and main result}

Many important problems in discrete mathematics deal with packing
structures with some desired property into a larger structure, and
their goal is typically to find as many disjoint structures with the
desired property as possible. Several classical results in
combinatorial optimization fit into this general framework. For
instance, the maximum matching problem can be seen as packing
vertex-disjoint edges. We also highlight Tutte's~\cite{Tutte61} and
Nash-Williams's~\cite{NashWilliams61} results on packing spanning
trees, as well as Menger's~\cite{Menger27} and
Mader's~\cite{Mad78} results on packing paths. See the book by
Cornu{\'e}jols~\cite{Co01} for many more examples.

Given the extensive literature on this topic, it is only natural that
there is a great number of packing results in extremal combinatorics
and random structures. For instance, the problem of packing
Hamiltonian cycles in random structures has been studied quite
intensively since the 1980s (see~\cite{SoKrSu11, Bo84,BoFr85, FrKr08,
  KnKuOs15, KrSa12, KuOs14}). In the particular case of digraphs, some
significant results have been obtained recently
(see~\cite{FekwSu17, FL,FeNeNoPeSk14}).

 Recently, Gao,
{P\'{e}rez-Gim\'{e}nez} and the third author~\cite{GPS14-RSA,GPS14} obtained
results concerning packing spanning trees in random graphs. As usual,
given a function $p \colon \mathbb{N} \rightarrow [0,1]$ and a
positive integer $n$, we let $\calg(n,p)$ be the random graph on
$[n]=\{1,\dotsc, n\}$ such that each edge is included independently
with probability~$p$. Moreover, given a sequence of probability spaces
$(\Omega_i, \calf_i, \Pr_i)_{i\in \mathbb{N}}$, we say that a sequence
of events $(A_i)_{i\in\mathbb{N}}$ holds \emph{asymptotically almost
  surely} (\aas{} for short) if $\Pr_n(A_n)\to 1$ as $n\to\infty$.
\begin{theorem}[Pu--P{\'e}rez-Gim{\'e}nez--Sato\footnote{The result stated here is not the
  strongest result obtained in~\cite{GPS14-RSA,GPS14}.}~\cite{GPS14-RSA,GPS14}]\label{thm:treepack}
  For $p = p(n)\in [0,1]$, the maximum number of
  edge-disjoint spanning trees in $\calg(n,p)$ is  \aas{}
$$\min\set[\big]{
      \delta(\calg(n,p)),
      \floor[\big]{m(\calg(n,p))/(n-1)}}.
  $$
\end{theorem}
It is easy to see that $\delta(\calg(n,p))$ and
$\floor[\big]{m(\calg(n,p))/(n-1)}$ are upper bounds for the number of
edge-disjoint spanning trees since every spanning tree has at least
one edge incident to every vertex and has exactly $n-1$ edges. The
following classical result proved by Tutte and Nash-Williams is the
main tool in~\cite{GPS14-RSA,GPS14} to prove that the maximum is achieved by one
of these two parameters.

\begin{theorem}[Tutte~\cite{Tutte61} and Nash-Williams~\cite{NashWilliams61}]\label{thm_NW}
Given a graph $G = (V,E)$ and an integer $k\geq 0$, $G$ contains $k$
  edge-disjoint spanning trees if and only if, for every partition of
  $V$ with $\ell$ parts, the number of edges with ends in different
  parts is at least $k(\ell-1)$.
\end{theorem}
It is quite natural that this result (which is actually a min-max
relation) can be successfully used for random graphs since the
partition condition is essentially an expansion condition and random
graphs are well known to have nice expansion properties.

Our main result is an analogue of Theorem~\ref{thm:treepack} for
digraphs. A \emph{digraph} $D=(V,A)$ is given by its finite set $V$ of
\emph{vertices} and its set $A \subset \{(u,v) \in V^2 \colon u \neq
v\}$ of \emph{arcs}. We say that an arc $(u,v)$ \emph{leaves} $u$ and
\emph{enters} $v$, or, alternatively, that it points at $v$. The
\emph{underlying graph} of a digraph $D=(V,A)$ is the graph (actually
multigraph) obtained by ignoring orientations on arcs. Our result
deals with packing arborescences, which are an analogue of spanning
trees in digraphs. Indeed, an \emph{arborescence} of a digraph is a
spanning sub-digraph such that its underlying graph is a rooted tree
and each vertex except the root has in-degree $1$ and the root has
in-degree zero. Roughly speaking, an arborescence is a spanning tree
with the arcs ``pointing away'' from the root. Let $\Dnp$ denote the
random digraph on $[n]=\{1,\dotsc, n\}$ such that each arc is included
independently at random with probability~$p$. Let $\uptau(\Dnp)$
denote the maximum number of arc-disjoint arborescences in $\Dnp$. For
every digraph $D$ and $v\in D$, let the \emph{in-degree} $\din_D(v)$
of $v$ be the number of arcs entering $v$ in $D$, while the
\emph{out-degree} $\dout_D(v)$ of $v$ is the number of arcs leaving
$v$ in $D$. Our main result may be stated as follows.
\begin{theorem}\label{Th:MainResult}
  For every $p = p(n) \in [0,1]$, the maximum number of arc-disjoint arborescences in $\Dnp$ \aas{} satisfies
\begin{align}
  \uptau(\cald(n,p)) 
  =
  \lambda(\Dnp),
\end{align}
where $\lambda(\Dnp)$ is the maximum integer $\lambda\geq 0$ such
  that, for all $0\leq \ell\leq \lambda$,
\begin{equation}
  \label{eq:lambdadef}
  \sum_{i=0}^{\ell-1} (\ell-i)|\{v\colon \din_{\Dnp}(v) = i\}|
  \leq 
  \ell.
\end{equation}
Moreover,
\begin{itemize}
\item[(a)] if $p = (\log(n)-h(n))/(n-1)$ with $h(n) = \omega(1)$, then
  $\lambda(\Dnp) = 0$ \aas;

\item[(b)]  if $p = (\log(n)+h(n))/(n-1)$ with
$h(n) = O(\log \log n)$, then $\lambda(\Dnp) \in
\{\deltain,\deltain+1\}$ \aas;
\item[(c)]  if $p = (\log(n)+h(n))/(n-1)$ with $h(n) = o(\log n)$ and
$h(n) = \Omega(\log \log n)$, then $\lambda(\Dnp) \sim \deltain$ \aas
\end{itemize}
\end{theorem}

One interesting feature of our result is that $\uptau(\Dnp)$ has a
very strong relation with the number of vertices with low degrees.
This differs from the graph case in the following sense.
Theorem~\ref{thm:treepack} tells us that, for random graphs, the
obstacles to pack spanning trees are quite simple: either we do not
have enough edges to get more spanning trees or we exhausted the edges
incident with a vertex. Our result shows that for random digraphs the
obstacles to pack arborescences are more intricate while still arising
from natural constraints. This is due to the fact that the root of an
arborescence plays a special role, which does not happen for
undirected graphs. In our case, the reason why $\lambda(\Dnp)$ is an
upper bound for $\uptau(\cald(n,p))$ is that, in order to pack $\ell$
arborescences, every vertex of $\Dnp$ whose in-degree is $\ell-i$ must
be the root of at least $i$ arborescences since its in-degree would be
exhausted. Quite interestingly, our condition does not involve the
out-degrees.

  Similarly to the undirected case, the core of our proof relies on a
  result on combinatorial optimization, which was proved by
  Frank~\cite{AF79} and is an analogue of Theorem~\ref{thm_NW}
  for digraphs. Instead of dealing with partitions, Frank's
  result imposes conditions on subpartitions. A \emph{subpartition} of
  a set $S$ is a collection of pairwise disjoint non-empty subsets of $S$.
  Note that, unlike a partition, a subpartition does not need to
  include every element of~$S$. For every digraph $D=(V,A)$ and
  $S\subseteq V$, let $\din_D(S)$ denote the number of arcs entering
  $S$ (from $V\setminus S$). For future reference, let also
  $\dout_D(S)$ be the number of arcs leaving $S$ (to
    $V\setminus S$) in $D$.

\begin{theorem}[Frank~\cite{AF79}]\label{teo:Frank}
  Let $D = (V,A)$ be a digraph and let $k\geq 0$ be an integer. Then
  $D$ contains $k$ arc-disjoint arborescences if, and only if, for
  every subpartition $\calp$ of~$V$, we have
  \begin{equation}
    \label{eq:subpart}
    \sum_{U\in \calp}\din_D(U) \geq k(|\calp|-1).
  \end{equation}
\end{theorem}
One of the difficulties of working with subpartitions instead of
partitions is that some vertices may be not included in any part and
the arcs entering such vertices do not contribute to the summation
in~\eqref{eq:subpart}, which is something that did not occur in the
graph case.

In terms of previous results about arborescences in random digraphs, Bal, Bennett, Cooper, Frieze, and
Pra{\l}at~\cite{BBCFP} have proved that in the random digraph
process (where the arcs are added one-by-one), the digraph contains an
arborescence a.a.s. in the step where there is single a vertex with
in-degree zero\footnote{This is part of their main result about
  rainbow arborescences.}. 

\noindent\textbf{Organization of the paper.} This paper is organized
as follows. In Section~\ref{sec:def}, we introduce the main
definitions and notation used in the paper. In Section~\ref{sec:prop},
we present the main properties of $\Dnp$ that are used: in
Section~\ref{ssec:degree} we study properties of the degrees in
$\Dnp$; in Section~\ref{ssec:lambda} we show the relation between
$\lambda(\Dnp)$ and the minimum in-degree; and in
Section~\ref{ssec:expansion} we prove a few basic expansion properties
of $\Dnp$. Finally, in Section~\ref{sec:main} we combine the results
from Section~\ref{sec:prop} with the result by Frank (Theorem~\ref{teo:Frank})
to complete the proof of Theorem~\ref{Th:MainResult}.

\section{Definitions and notation}
\label{sec:def}
In this section, we define the main concepts used in this paper. We
will repeat a few definitions already presented in the introduction so
that the reader can easily find any of them.

\begin{definition} (Random digraph $\Dnp$) Given a function
  $p=p(n)\colon\mathbb{N}\to [0,1]$, let $\Dnp$ denote the random
  digraph with vertex set $[n]=\{1,2,\dotsc, n\}$ such that each of
  the $n(n-1)$ arcs is included independently at random with
  probability $p$.
\end{definition}

\begin{definition} (Neighbourhoods and degrees) Given a digraph
  $D=(V,A)$ and $v\in V$, we define the \emph{in-neighbourhood}
  of~$v$, denoted by $\Gamma_{D}^{in}(v)$, as the set $\{u\in V\colon
  (u,v)\in A\}$. Similarly, we define the \emph{out-neighbourhood}
  of~$v$, denoted by~$\Gamma_{D}^{out}(v)$, as the set $\{u\in V\colon
  (v,u)\in A\}$. Moreover, we define the \emph{in-degree} of $v$ as
  $\din_D(v)=|\Gamma_{D}^{in}(v)|$ and the \emph{out-degree} of $v$ as
  $\dout_D(v)=|\Gamma_{D}^{out}(v)|$. That is, $\din_D(v)$ is the
  number of arcs ``entering''~$v$ and $\dout_D(v)$ is the number of
  arcs ``leaving''~$v$. Let $\deltain(D) = \min_{v\in V}\din_D(v)$ and
  $\deltaout(D) = \min_{v\in V}\dout_D(v)$.

\end{definition}

\begin{definition}(Cuts) Given a digraph $D = (V,A)$ and disjoint sets
  $S, S'\subseteq V$, we define $A_D(S,S')$ as the set of arcs
  $(u,v)\in A$ such that $u\in S$ and $v\in S'$.
\end{definition}

\begin{definition}(Induced digraphs) Given a digraph $D = (V,A)$ and
  $S\subseteq V$, we define $D[S]$ as the digraph with vertex set $S$
  with edge set $A_D[S] = \{(u,v)\in A\colon u\in S,\ v\in S\}$.
\end{definition}

\begin{definition} (Arborescences) An \emph{arborescence} of a digraph
  $D=(V,A)$ is a digraph $T=(V,A_T)$ where $A_T\subseteq A$ such that
  the underlying graph of $T$ is tree and each vertex except the root
  has in-degree $1$ and the root has in-degree zero. Let $\uptau(D)$
  denote the maximum number of arc-disjoint arborescences in $D$.
\end{definition}

\begin{definition}
  Given a digraph $D=(V,A)$, let $\lambda(D)$ denote the maximum
  integer $\lambda\geq 0$ such that, for all $0\leq \ell\leq \lambda$,
\begin{equation}
  \label{eq:lambdadef}
  \sum_{i=0}^{\ell-1} (\ell-i)|\{v\colon\din_{D}(v) = i\}|
  \leq 
  \ell.
\end{equation}
\end{definition}

We use $\din(v)$ to denote $\din_{\Dnp}(v)$ for ease of notation.
Similarly, $\dout(v) = \dout_{\Dnp}(v)$, $\deltain = \deltain(\Dnp)$,
$\deltaout = \deltaout(\Dnp)$, $\tau = \tau(\Dnp)$ and $\lambda =
\lambda(\Dnp)$, and so on.

In all results in this paper, except stated otherwise, the probability
space is the one defined by $\Dnp$ and the asymptotics are for $n$
going to infinity. We use standard asymptotic notation, which may be found in~\cite[Section 1.2]{JLR00}. 

In many proofs, we will use the well known subsubsequence principle, which states that, if $x$ is a constant and $(x_n)$ is a real sequence whose subsequences
always have a subsubsequence converging to $x$, then $x_n \to x$.

\section{Properties of the random digraph $\Dnp$}
\label{sec:prop}

In this section, we study the behaviour of the degrees in $\Dnp$. We
also prove some simple properties about cuts in $\Dnp$. In
Section~\ref{ssec:binom}, we state two basic results on binomial
random variables that are used throughout the paper. For basic probabilistic
 results (such as Markov's and Chebyshev's inequality), we refer the reader to Alon and Spencer~\cite{AlSp03}.

\subsection{Properties of binomial random variables}
\label{ssec:binom}
In this section, we state two results on binomial random variables.

\begin{theorem}[Chernoff's bounds~\cite{JLR00}]\label{teo:chernoff}
  Let $X_1,\ldots,X_n$ denote $n$ independent Bernoulli variables. Let
  $X = \sum^n_{i=1}X_i$ and let $\mu = \bbe[X]$. Then, for any $0 <
  \tau <1$,
  \begin{align}
    \Pr (X\geq (1+\tau)\mu) & \leq e^{-\tau^2\mu/3},\\
    \Pr (X\leq (1-\tau)\mu) & \leq e^{-\tau^2\mu/2}.
  \end{align}
\end{theorem}

\begin{lemma}[Lemma 16~\cite{GPS14-RSA,GPS14}]\label{lemma:lowerTailBinom}
  For every constant $\eta > 0$ there exist positive constants $C_1$
  and $C_2$ such that the following holds for any function $0 \leq p =
  p(n) \leq 1/\sqrt{n}$ and every integer $0 < k \leq (1-\eta)np$. Let
  $X \sim \bin(n,p)$. Then,
  \begin{equation}
    \Pr(X \leq k) =
    C\left(\frac{e^{-pn}}{\sqrt{k}}\right)
    \left(\frac{epn}{k}\right)^k,
    \text{ with $C_1 \leq C \leq C_2$}.
  \end{equation}
\end{lemma}

\subsection{Degrees in $\Dnp$}
\label{ssec:degree}

In this section, we present some results on the minimum in-degree and
out-degree in $\Dnp$. We also prove some properties of vertices with
low degree.

The following lemma is an application of
Lemma~\ref{lemma:lowerTailBinom} to the in-degrees and
out-degrees of $\Dnp$. 
\begin{lemma}\label{lemma:minDegreeSparse} Let $0 < \eta < 1$ be a constant. There exist constants $C_1>0$ and
  $C_2>0$ such that, for any function $\alpha=\alpha(n) \in(0,
  1-\eta]$ and any function $p$ satisfying $0.9 \log n/(n-1) \leq p
  \leq 1/\sqrt{n}$, the following holds\footnote{Here we
    use $d^{in/out}(v)$ to denote either $\din(v)$ or $\dout(v)$.}:
  \begin{enumerate}[label=(\roman*)]
  \item\label{lema:degree} 
    There exists $C=C(n)\in[C_1,C_2]$ such that, for every $v\in [n]$,
    $$\Pr\left(d^{in/out}(v) \leq \alpha
      p(n-1)\right) = \frac{C}{\sqrt{\alpha p(n-1)}} \exp
    \left(-p(n-1)\left(1-\alpha\log\left(\frac{e}{\alpha}\right)\right)
    \right).$$
  \item\label{lema:minDegree1} $\Pr\left(\delta^{in/out} 
      \leq \alpha p(n-1)\right) 
    \leq 
    \displaystyle\frac{C_2}{\sqrt{\alpha p (n-1)}}
    \exp \paren[\Big]{\log n -
      p(n-1)\left(1-\alpha\log\left(e/\alpha \right)\right)}$;
  \item\label{lema:minDegree2} $\Pr\left(\delta^{in/out} > \alpha
      p(n-1)\right) \leq
    \displaystyle{\frac{\sqrt{\alpha p (n-1)}}{C_1}}
    \exp\paren[\Big]{p(n-1)\left(1-\alpha\log\left(e/\alpha\right)\right)
      - \log n}$.
  \item\label{eq:smalldindout} 
    $\Pr\paren[\Big]{\exists v\in[n] \text{ s.t. }\min\{\din(v), \dout(v)\}\leq \alpha p(n-1)}$
      
\hspace{5cm}${\displaystyle\leq {\frac{(C_2)^2}{\alpha p(n-1)}}\exp\left(\log n -2p(n-1)
        \left(1-\alpha\log\left(\frac{e}{\alpha}\right)
        \right)\right)}$.    
  \end{enumerate}

\end{lemma}
\begin{proof}
  The proof of (i)--(iii) is basically the same as the proof
  of~\cite[Lemma 18]{GPS14-RSA,GPS14}. We include it here for the sake of
  completeness.

  For every $v\in V$, the in-degree $\din(v)$ of $v$ has distribution
  $\bin(n-1,p)$. Thus, by Lemma~\ref{lemma:lowerTailBinom}, there
  exist $C_1$ and $C_2$ (depending only on $\eta$) and a constant
  $C=C(n) \in [C_1, C_2]$ such that
\begin{align}
  \Pr\left(d^{in}(v) \leq \alpha p(n-1)\right) & = C
  \left(\frac{e^{-p(n-1)}}{\sqrt{\alpha p(n-1)}}\right)
  \left(\frac{e}{\alpha}\right)^{\alpha p(n-1)}\nonumber\\
  & = \frac{C}{\sqrt{\alpha p(n-1)}}
  \exp\left(-p(n-1)\left(1-\alpha\log\left(\frac{e}{\alpha}\right)\right)
  \right).\label{eq:lemmaMinDegree1}
\end{align}
This proves~\ref{lema:degree}. Let $Y$ denote the number of
vertices $v\in [n]$ such that $\din(v)\leq \alpha p(n-1)$. Then, by~\eqref{eq:lemmaMinDegree1}, we conclude that
$$\bbe(Y) \geq 
  {\frac{C_1}{\sqrt{\alpha p(n-1)}}}
  \exp\left(\log n -p(n-1)\left(1-\alpha\log\left(\frac{e}{\alpha}\right)\right)
  \right)$$
and
\begin{equation}
  \label{eq:lemmaMinDegree2}
  \bbe(Y) \leq 
  {\frac{C_2}{\sqrt{\alpha p(n-1)}}}
  \exp\left(\log n -p(n-1)\left(1-\alpha\log\left(\frac{e}{\alpha}\right)\right)
  \right).
\end{equation}
Thus, \ref{lema:minDegree1} follows by
applying Markov's inequality. Since the in-degrees of distinct vertices are
independent random variables, $Y$ is a binomial random variable with
probability $p'$ given by~\eqref{eq:lemmaMinDegree1}. Thus, $\var(Y) =
np'(1-p') \leq \bbe{Y}$ and so, by Chebyshev's inequality, we obtain
\begin{align*}
  \Pr(Y= 0) &
  \leq \frac{\var(Y)}{(\bbe Y)^2} \leq \frac{1}{\bbe{Y}}
   \leq {\frac{\sqrt{\alpha p(n-1)}}{C_1} }
  \exp\left(-\log n +
    p(n-1)\left(1-\alpha\log\left(\frac{e}{\alpha}\right)\right)\right).
\end{align*}
Proving (i), (ii) and (iii) for $d^{out}$ and $\deltaout$ is analogous.

For any $v\in [n]$, \ref{eq:smalldindout} follows trivially
from~\eqref{eq:lemmaMinDegree2} (and its analogue for $\dout(v)$) and
the fact that $\din(v)$ and $\dout(v)$ are independent random
variables.
\end{proof}

  \begin{corollary}\label{lem:mindeg_sparse_o}
    Let $\alpha \in (0,1)$ and let $h(n)$ be a function such that $h(n)=o( \log n)$. Let $p = p(n) =
    (\log(n)+h(n))/(n-1)$. Then there is $C>0$ such that $\deltain < \alpha \log{n}$ with probability at least $1-n^{-C}$.
\end{corollary}
\begin{proof}
  It suffices to prove the result for $h(n)\geq 0$. Let
  $\alpha\in(0,1)$ be a constant. Note that $\beta :=
  1-\alpha\log\left(e/\alpha\right)$ is a constant less than~$1$.
  Lemma~\ref{lemma:minDegreeSparse}\ref{lema:minDegree2} leads to
  \begin{equation*}
    \Pr(\deltain > \alpha p(n-1))
    \leq  
    {\displaystyle{\frac{\sqrt{\alpha p (n-1)}}{C_1}}}
    \exp\paren[\Big]{\beta (\log n+h(n))- \log n} = o(n^{(\beta-1)/2}).
  \end{equation*}
\end{proof}
 
 \begin{lemma}\label{lem:mindeg_sparsest_new}
  Let $p = p(n) = (\log(n)+h(n))/(n-1)$ be such that $h(n) \leq C' \log
  \log n$, where $C'$ is a positive constant. Then, for any constant $C>C'$,  $\deltain \leq C$ \aas 
  \end{lemma}
  \begin{proof}
    The expected number of  vertices with in-degree~$C$ in~$\cald(n, p)$ is
    \begin{equation*}
    	\begin{split}
    	n\binom{n-1}{C}p^C(1-p)^{n-C-1}  & \geq \exp\left(\log n +C\log\log n - C\log C -(n-1)p + o(1)\right) \\
    														& \geq \exp\left(-h(n) + C\log\log n- C\log C + o(1)\right)=\omega(1).
    	\end{split}
    \end{equation*}
    As in the proof of Lemma~\ref{lemma:minDegreeSparse}\ref{lema:minDegree2}, we may apply Chebyshev's inequality to conclude that the minimum in-degree is at most~$C$ \aas
  \end{proof}

It turns out that the following real function appears often in our paper:
\begin{equation}\label{def_F}
F(x) = 1- x \log{\frac{e}{x}} =  1-x + x\log x. 
\end{equation}
This is a continuous and strictly decreasing function for $x\in(0,1)$. Moreover,  $F(1) = 0$ and $\lim_{x\to 0^+}F(x) = 1$.  In particular, for every $\phi>1$, there exists a single $\alpha \in (0,1)$ such that $F(\alpha) = 1/\phi$.
  
\begin{corollary}
  \label{lem:mindeg_concentration}
  Let $\phi>1$ be a constant. Let $\alpha \in(0,1)$ be such that
  $F(\alpha)=1-\alpha + \alpha\log \alpha = 1/\phi$. For $p=p(n)\sim \phi \log
  (n)/(n-1)$, we have  $\deltain \sim \alpha p(n-1)$ and $\deltaout \sim
  \alpha p(n-1)$ \aas 
\end{corollary}
\begin{proof}
  Let $\gamma = \gamma(n)$ be such that $\gamma(n)\log n = p(n-1)$.
  Thus, $\gamma\sim\phi$. Let $\eps
  \in(0,\min\{\alpha,1-\alpha\})$ be a constant.

  We have $\lim_{n\to\infty} \gamma(n) F(\alpha+\eps) < F(\alpha)
  \lim_{n\to\infty}\gamma(n) = 1$ and then by
  Lemma~\ref{lemma:minDegreeSparse}\ref{lema:minDegree2}, we obtain
  $\deltain \leq (\alpha+\eps)p(n-1) $ with probability going to~$1$.
  On the other hand, since ${\lim_{n\to\infty}(\gamma(n) F(\alpha-\eps))
    > 1}$, by Lemma~\ref{lemma:minDegreeSparse}\ref{lema:minDegree1},
  we have $\deltain\geq (\alpha-\eps)p(n-1)$ a.a.s. The proof for
  $\deltaout$ is similar.
\end{proof}

\begin{definition}
   We say that a vertex $v\in [n]$ is \emph{$\eps$-in-light}
    if $\din(v) \leq \deltain + \eps np$ and
    \emph{$\eps$-out-light} if $\dout(v) \leq \deltaout + \eps
    np$.
\end{definition}


\begin{lemma}\label{lemma:light_vertex}
  For every constant $\varphi\geq 0.9$, there exists $\eps>0$ such that the following holds for sufficiently large $n$.
  For any $0.9\log n/(n-1) \leq p
  \leq \varphi\log n/(n-1)$, with probability at least $1-n^{-0.18}$,
  there is no pair $(u,v)$ of $\eps$-in-light vertices such that
  $uv\in A$ or $\Gamma^{in}(v)\cap\Gamma^{in}(u)\neq \varnothing$
  and there is no pair $(u,v)$ of $\eps$-out-light vertices
    such that $uv\in A$ or $\Gamma^{out}(v)\cap\Gamma^{out}(u)\neq
    \varnothing$.
\end{lemma}
\begin{proof}
  We claim that it is possible to choose $\alpha$ and $\eps$ so that, for large $n$, we have
  \begin{equation}
    \label{eq:prob_light}
    \Pr(\delta^{in} > \alpha(n-1)p) \leq n^{-0.19}
    ~\text{and}~
    \Pr(\din(v)\leq (\alpha+\eps) p(n-1)) \leq n^{-0.7},\ \forall v\in V.
  \end{equation}
  Assuming that this claim holds, and conditioning on the event that $\delta^{in} \leq \alpha(n-1)p$, let $S$ be the set of vertices $v\in [n]$ such that $\din(v)\leq
  (\alpha+\eps)p(n-1)$. Note that $S$ contains all $\eps$-in-light vertices. Then, for any vertices $u, v\in V$, by our choice
  of $\alpha$ and $\eps$, we have 
  \begin{align*}
    \Pr\left(
      uv \in A, u\in S \text{ and } v\in S
      \right)
      &= p\Pr(v\in S | uv\in A) \Pr(u\in S | uv\in A)
      \\
      &\leq (1+o(1)) p n^{-1.4}.
  \end{align*}
  Since we have at most $n(n-1)$ choices for $(u,v)$ and $p\leq
  \varphi \log n/(n-1)$, the expected number of pairs of adjacent
  vertices in $S$ is $(1+o(1))\varphi n^{-0.4}\log n$. Thus, the
  probability that there are adjacent $\eps$-in-light vertices is at most
  $n^{-0.19}+(1+o(1))\varphi n^{-0.4}\log n \leq \frac14 n^{-0.18}$ for
  sufficiently large $n$.

  For any vertices $u, v, z\in V$, by our choice of $\alpha$ and
  $\eps$, we have
  \begin{align*}
    \Pr\left(
      zu,zv \in A, u\in S \text{ and } v\in S
      \right)
      &= p^2\Pr(v\in S | zv\in A) \Pr(u\in S | zu\in A)
      \\
      &\leq (1+o(1)) p^2 n^{-1.4}.
  \end{align*}
  Hence, since we have at most $n(n-1)(n-2)$ choices for $(u,v,z)$ and
  $p\leq \varphi \log n/(n-1)$, the expected number of pairs of
  adjacent $\eps$-in-light vertices is $(1+o(1))\varphi n^{-0.4}\log^2 n$ and
  the result follows as above. Note that the same argument applies for the out-degree.

  Finally we show how to choose $\alpha$ and $\eps$ to obtain the
  desired bounds for the probabilities in~\eqref{eq:prob_light}. To this end, let
  $F(x) = 1-x+x\log{x}$ be the function defined in~\eqref{def_F}, and let $\beta$ be the constant such that $\beta \sim (n-1)p/\log n$. 
	By hypothesis, $0.9 \leq \beta \leq \varphi$. Choose $\alpha$ so that $\beta F(\alpha) = 0.81$.
	Then the RHS of
  Lemma~\ref{lemma:minDegreeSparse}\ref{lema:minDegree2} is at most
  $O(1) \exp(-0.2\log n+(1/2)\log \log n) \leq n^{-0.19}$.
  We can then
  choose $\eps>0$ so that $\beta F(\alpha+\eps) = 0.71$ and the RHS
  of Lemma~\ref{lemma:minDegreeSparse}\ref{lema:degree} becomes
  $O(1)\exp(-\beta\log nF(\alpha+\eps) - (1/2)\log\log n) \leq \exp(-0.7\log n-(1/2)\log \log n) \leq n^{-0.7}$.
\end{proof}

The next result follows immediately from Chernoff's inequality 
(Theorem~\ref{teo:chernoff}).
\begin{lemma} 
\label{lemma:concentration_degrees}
Let $\psi = \psi(n) = \omega(1)$ and $\phi=1/\sqrt{\log{n}}$. Let $p = p(n)$ be such that
$(n-1)p=\psi\log n$ \aas{}\ for all $v\in [v]$ we have $(1-\phi) p(n-1)\leq
\din(v)\leq (1+\phi)p(n-1)$ and $(1-\phi) p(n-1)\leq \dout(v)\leq
(1+\phi)p(n-1)$.
\end{lemma}

For a function $p \colon \mathbb{N}
\to [0,1]$ and integers $n,k \in \mathbb{N}$, we consider the random
variable $Y_k=Y_k(\Dnp)$ that counts the number of vertices of
in-degree $k$ in $\Dnp$.
 
\begin{definition}\label{def:deltastar}
Let $\deltastar = \deltastar(\Dnp)$ denote
  the minimum integer $k\geq 0$ such that $\mean{Y_k}\geq 1$ in
  $\Dnp$.
\end{definition}

The following result shows the relation between $\deltastar$
  and $\deltain$.
  \begin{lemma}\label{lemma:min_deltastar}
    Let $p = p(n) \sim\log(n)/(n-1)$. Then $\deltain
    \in\{\deltastar-1,\deltastar, \deltastar+1\}$ \aas{} Moreover, we have
    that, if $\mean{Y_{\deltastar-1}}=o(1)$, then $\deltain\in\{\deltastar,
    \deltastar+1\}$ and, if $\mean{Y_{\deltastar-1}}=\Omega(1)$, then
    $\deltain\in\{\deltastar-1, \deltastar\}$.
  \end{lemma}
  \begin{proof}
It is straightforward to show that there exists $k^*(n) = o(\log n)$ such that $\mean{Y_{k^*(n)}} = \omega(1)$.  Thus, $\deltastar = o(\log
    n)$. For every~$k =
    o(\log n)$, we have 
    \begin{equation}
      \label{eq:ratio}
      \mean{Y_{k-1}}/\mean{Y_k} \sim
      k/((n-1)p)      
    \end{equation}
    First assume that $\mean{Y_{\deltastar-1}}=o(1)$.   By Markov's inequality and by~\eqref{eq:ratio}, we have
    $\deltain \geq \deltastar$ \aas\ 

    Again by~\eqref{eq:ratio}, we obtain $\mean{Y_{\deltastar+1}} \sim
    np/\deltastar\cdot \mean{Y_{\deltastar}} =\omega(1)$. Hence, by
    Chebychev's inequality (using the independence of the in-degrees),
    we have $\Pr(Y_{\deltastar+1}=0)=o(1)$, so that $\deltain \leq \deltastar+1$ \aas\ The proof for the
    case $\mean{Y_{\deltastar-1}}=\Omega(1)$ is similar.
    
  \end{proof}

\subsection{Estimating $\lambda(\Dnp)$}
\label{ssec:lambda}

In this section, we give tight estimates for $\lambda=\lambda(\Dnp)$ depending on
the range of~$p$. 

The easiest case is $p=\omega(\log{n}/n)$, when our estimate follows immediately from
Lemma~\ref{lemma:concentration_degrees} and the fact that $\lambda$ is
between the minimum and the maximum in-degree.
\begin{corollary}
\label{corollary:lambda_dense}
  If $\phi=\phi(n) = \omega(1)$ is a function and $p=p(n)\sim \phi
  \log (n)/(n-1)$, then $\lambda(n) \sim p(n-1)$ \aas 
\end{corollary}

Next we consider other ranges of $p$. 
\begin{lemma}\label{lemma:LambdaUpcriticalConst}
  Let $\phi>1$ be a constant. If $p=p(n)\sim \phi \log (n)/(n-1)$, then
  $\lambda \sim \deltain$.
\end{lemma}
\begin{proof}
  By Corollary~\ref{lem:mindeg_concentration}, for $\alpha\in(0,1)$
  such that $F(\alpha) = 1/\phi$, we have
  $\deltain\sim \alpha p(n-1)$ a.a.s. Given $\eps > 0$, Lemma~\ref{lemma:minDegreeSparse}\ref{lema:degree} ensures that there is $0 <\beta < 1$ (depending on $\alpha$ and $\eps$) such that, for any
  vertex~$v$, we have  \begin{equation*}
    \Pr(\din(v)\leq\alpha{(1+\eps)}p(n-1))
    = \Theta\left(
       n^{-\beta}
      \right).
    \end{equation*}
    Since the in-degrees of distinct vertices are independent, we may
    apply Chernoff's inequality (Theorem~\ref{teo:chernoff}) to the
    binomial random variable counting the number of vertices whose
    in-degree is at most $\alpha{(1+\eps)}p(n-1)$ to conclude that there
    are $\Theta(n^{1-\beta})$
    such vertices \aas \
    Since $n^{1-\beta}  =\omega(\alpha{(1+\eps)}p(n-1))$, this
      implies that $\lambda \leq \alpha
{(1+\eps)}p(n-1) \sim
    (1+\eps)\deltain$ \aas{} Since $\lambda \geq \deltain$ holds
    trivially, our result follows.
\end{proof}

In the next cases, we use the following simple fact.
\begin{claim}\label{claim_aux}
Let $D$ be a digraph and $k \geq 0$ an integer. If $Y_k > k + 1$, then $\lambda(D) \leq k$.
\end{claim}

\begin{lemma}
\label{lemma:lambda_sparsest} Let $p = p(n) = (\log(n)+h(n))/(n-1)$ be
such that $h(n) = O(\log\log n)$. Then $\lambda(\Dnp) \in \{\deltain,
\deltain+1\}$ a.a.s.
\end{lemma}
\begin{proof}
Consider $\deltastar$ from Definition~\ref{def:deltastar}. First assume that
  $\mean{Y_{\deltastar-1}} = o(1)$. We shall prove that $\lambda \leq \deltastar+1$ \aas{}, which leads to the desired conclusion because
  $\lambda \geq \deltain$ and Lemma~\ref{lemma:min_deltastar} ensures that $\deltain \in \{\deltastar, \deltastar+1\}$ \aas{}.
The definition of
    $\deltastar$ implies that $\mean{Y_\deltastar} \geq 1$, and
    Lemma~\ref{lem:mindeg_sparsest_new} implies that there exists~$C>0$ such
    that~$\deltain \leq C$ \aas, and hence
    $\deltastar = O(1)$. Moreover, $\mean{Y_{\deltastar+1}} \sim np/\deltastar\cdot \mean{Y_{\deltastar}}$, so that $\mean{Y_{\deltastar+1}} \geq np/\deltastar(1+o(1)) =
    \omega(1)$. By Chernoff's inequality (Theorem~\ref{teo:chernoff}),
    $\Pr(Y_{\deltastar+1}\leq \mean{Y_{\deltastar+1}}/2) \leq \exp(-A\log n))$, where
    $A >0$. Thus, $Y_{\deltastar+1}> \deltastar+2$ \aas{} and so $\lambda\leq
    \deltastar+1$ \aas{} by Claim~\ref{claim_aux}.

The proof is similar for the case $\mean{Y_{\deltastar-1}} =
\Omega(1)$, where we prove that $\lambda\leq \deltastar$ \aas{} The result then follows by the subsubsequence principle.
\end{proof}

\begin{lemma} 
\label{lem:lambda_critical}
Let $p = p(n) = (\log(n)+h(n))/(n-1)$ be such that $h(n)
  = o(\log n)$ and $h(n) = \omega(\log \log n)$. Then $\lambda(\Dnp)
  \sim \deltain$ a.a.s.
\end{lemma}
\begin{proof}
The proof is very similar to the proof for Lemma~\ref{lemma:lambda_sparsest}. Let $\eps>0$ be a constant and fix $T =
  \lfloor(1+\eps)\deltain\rfloor$. By Lemma~\ref{lemma:minDegreeSparse}\ref{lema:minDegree1}, we have $T>\deltain$ \aas{} We will address only
  the case where $\mean{Y_{\deltastar-1}} = o(1)$. By
  Lemma~\ref{lemma:min_deltastar}, we have $\deltain \in \{\deltastar, \deltastar+1\}$ a.a.s. To show that $\lambda \leq T$, we prove that $Y_{T} > T+1$ a.a.s. By
  Corollary~\ref{lem:mindeg_sparse_o}, we have that $\deltain = o(\log n)$ and so $T = o(\log
  n)$ as well. Then, for $k=\deltain$, we have
  \begin{equation}
    \frac{\mean{Y_T}}{\mean{Y_k}}  
    \sim 
    \left(\frac{np}{k}\right)^{T-k}
    =
   \omega(T).
  \end{equation}
  Since $\mean{Y_k} =\Omega(1)$, we have that $Y_T = \omega(T)$ a.a.s.\
  by Chernoff's inequality (Theorem~\ref{teo:chernoff}),
  which implies that $\lambda\leq T$ a.a.s. 
\end{proof}

\subsection{Expansion properties}

In this section, we  investigate properties of the cuts of
$\Dnp$. We start by proving a simple result about the number of arcs
going from a ``large'' set to another ``large'' set of vertices.

\label{ssec:expansion}
\begin{lemma}\label{lema:expansion_UpperRange_big}
  Let $f(n)\to\infty$
 and let $\zeta$ be a positive constant.
  There exists a positive constant $C$ such that, for $p = p(n) \in [f(n)/n, 1]$ and large $n$, the probability that there exist
  disjoint sets $S,S'\subseteq [n]$ with size at least $\zeta n$ 
  such that $|A(S',S)| < \zeta^2 n^2 p/2$ is at most $n^{-C}$.
\end{lemma}

\begin{proof}
  Let  $S,S'\subseteq [n]$ be disjoint sets with size at
  least $\zeta n$. Then $|A(S,S')|$ has distribution
  $\bin(|S'||S|,p)$. Thus, $\bbe(| A(S',S)|) =|S||S'|p$ and by
  Chernoff's inequality (Theorem~\ref{teo:chernoff}), we have
\begin{equation}
  \Pr\left(| A(S',S)| \leq \frac{\zeta^2 n^2 p}{2}\right)
  \leq \exp\left(-\frac{\zeta^2 n^2 p}{8}\right).\label{eq:VA_arc}
\end{equation} 
By the union bound, the probability that there exist disjoint sets
$S,S'\subseteq [n]$ with size at least $\zeta n$ such that $|A(S',S)|
< \zeta^2 n^2 p/2$ is at most
\begin{align*}
  \sum_{s,s' \geq \zeta n} \binom{n}{s}\binom{n}{s'} \exp\left(-
    \frac{\zeta^2 n^2p}{4}\right) &\leq
  4^n  \exp\left(-\frac{\zeta^2 n^2p}{4}\right) 
  < \exp\left(2n- \frac{\zeta^2 n^2p}{4}\right),
\end{align*}
and the result follows since $np\geq f(n)\to\infty$.
\end{proof}

{Next we prove a lemma about the number of induced arcs in
  sets that are ``not too large''. Later we will use this lemma to argue that
  many arcs must leave such sets.}
\begin{lemma}\label{lema:expansion_UpperRange_small}
  Consider a function $f= f(n)\to\infty$, and let $\phi$ be positive constant. There
  exists a positive constant $\zeta$ such that, for $p = p(n) \in [f/n, 1]$ and large $n$, the probability that there
  exists $S \subseteq [n]$ with size $|S| \leq \zeta n$ such that
  $|A[S]| > \phi n |S|p$ is at most $e^{-f^2/2}$.
\end{lemma}

\begin{proof}
  Let $\zeta > 0 $ be sufficiently small so that $e\zeta/\phi \leq
  e^{-1/\phi^2}$. Let $S\subseteq [n]$ be a set with size $s \leq
  \zeta n$. If $s\leq \phi n p$, then $|A[S]| \leq s(s-1) \leq \phi n
  p s$. So assume $s\geq \phi n p$. Then the probability that $|A[S]|
  > \lceil \phi n p s\rceil$ is at most $\binom{s(s-1)}{\lceil \phi n p s\rceil} p^{\lceil \phi n p s\rceil}$.
  Thus, the expected number of sets $S$ with size at most $\zeta n$ with
  $|A[S]| > \phi p n |S|$ is at most
\begin{align*}
\sum_{\phi np \leq s \leq \zeta n}& 
\binom{n}{s}\binom{s^2}{\lceil\phi nsp\rceil}p^{\lceil\phi nsp\rceil}\leq
\sum_{\phi np \leq s \leq \zeta n} 
\left(\frac{ne}{s}\right)^s\left(\frac{se}{\phi n}\right)^{\lceil\phi
  nsp\rceil} \\
&= \sum_{\phi np \leq s \leq \zeta n} \left(\left(\frac{e^2}{\phi}\right)\left(\frac{se}{\phi n}\right)^{\frac{\lceil\phi nsp\rceil}{s}-1}\right)^s\\ 
&\leq \sum_{\phi np \leq s \leq \zeta n} 
\left(\left(\frac{e^2}{\phi}\right)\left(\frac{\zeta
      e}{\phi }\right)^{\frac{\lceil\phi
      nsp\rceil}{s}-1}\right)^s\quad\text{since
}s\leq \zeta n \text{ and }np\geq f\to\infty\\
&\leq
\sum_{\phi np \leq s \leq \zeta n}
\left(\left(\frac{e^2}{\phi}\right)\left(\frac{\zeta  e}{\phi }\right)^{\phi
    np-1}\right)^s \quad\text{since
}\zeta e/\phi < 1\\
&\leq \sum_{\phi np \leq s \leq \zeta n}\left(\beta e^{-\frac{np}{\phi}}\right)^s,
\quad\text{for }\beta =e/\zeta \text{ since }
\zeta e/\phi < e^{-1/\phi^2}
\\&\leq  2(\beta e^{-\frac{np}{\phi}})^{\phi np} \leq e^{-(np)^2/2}\leq e^{-f^2/2},
\end{align*}
for $n$ sufficiently large. The result then follows
by Markov's inequality.
\end{proof}

In the next lemma, we compare $\din(S)$ and $\deltain$ in the
range where $p = (1+\Omega(1))\log n/(n-1)$ for $S$ of size from $2$
to $n-2$.
\begin{lemma}\label{lema:expansion_minDegree}
  Let $\psi$ be a positive constant. There is a positive constant
  $C$ such that, for $p = p(n) \in[(1+\psi)\log
  n/(n-1),1]$ and large $n$, the probability that there exists a set $S \subset [n]$
  with $2 \leq |S| \leq n-2$ such that $d^{in}(S) < 1.5\delta^{in}$ is
  at most $n^{-C}$.
\end{lemma}
\begin{proof} Let $\gamma(n)$ be such that $p(n-1) = \gamma(n) \log
  n$. First assume that $\gamma(n) \geq 100$. Let $0 < \epsilon_1 < 5/2
  - \sqrt{6}$, $\tau = 2(1+\epsilon_1)/10$ and $\epsilon_2>0$ be such
  that $2(1-\tau)(1-\epsilon_1)\geq 1.5(1+\eps_2)$.
                
  Let $S\subseteq [n]$ and let $\bar{S} = [n]\setminus S$. Then
  $\din(S) = \dout(\bar{S}) = |A(\bar S, S)|$ which is distributed as
  $\bin(s \bar{s}, p)$ where $s = |S|$ and $\bar{s} = |\bar{S}|$. Then,
  by Chernoff's inequality (Theorem~\ref{teo:chernoff}) and the union
  bound, the
  probability that there exists a set $S \subset [n]$ with $2 \leq |S|
  \leq n-2$ such that $d^{in}(S) < (1-\tau)s\bar{s} p$ is at most
	\begin{align}
		\sum_{2 \leq s \leq n-2}\binom{n}{s}\exp\left(-\frac{\tau^2s\bar{s}p}{2}\right) 
			& \leq 2\sum_{2 \leq s \leq n/2}\binom{n}{s}\exp\left(-\frac{\tau^2s\bar{s}p}{2}\right) \nonumber \\
			& \leq 2 \sum_{2 \leq s \leq n/2} \exp\left(s + s\log n - s\log s - \frac{\tau^2|S||\bar S|p}{2}\right) \nonumber \\
			& \leq 2 \sum_{2 \leq s \leq n/2} \exp\left(s\log n\left( 1 + \frac{1 - \log s}{\log n} - \frac{\gamma\tau^2}{4}\right)\right) \label{step1}\\
			& \leq 2 \sum_{2 \leq s \leq n/2} \exp\left(-\epsilon_1 s\log n\right) \leq 4n^{-2\eps_1}. \label{step2}
	\end{align}
  To obtain~\eqref{step1}, we use $|S||\bar{S}| =s(n-s)\geq sn/2$, and to obtain~\eqref{step2}, we use $\gamma \tau^2/4 \geq (1+\eps_1)^2$.
        Note that, if $|S| \leq \epsilon_1 n$ or $|\bar{S}|\leq
        \epsilon_1 n$, by our choice of $\tau$, $\epsilon_1$ and $\epsilon_2$,
	\begin{equation}
          (1-\tau)|S||\bar S|p \geq (1-\tau)2(1-\epsilon_1)np \geq 
          1.5(1+\eps_2) np.
	\end{equation}
        If $|S| \geq \epsilon_1 n$ and $|\bar S|
        \geq \epsilon_1 n$, then for sufficiently large $n$,
	\begin{equation}
          (1-\tau)|S||\bar S|p \geq (1-\tau)\epsilon_1^2n^2p > 1.5(1+\eps_2) np.
	\end{equation}
       It is easy to see that $\deltain \leq (1+\eps_2)np$ with very large probability. Indeed, by Chernoff's inequality (Theorem~\ref{teo:chernoff}), the total number of arcs  
        in the random digraph satisfies $|A| \geq
        (1+\eps_2)n(n-1)p$ with probability at most $\exp(-\eps_2^2
        n(n-1)p/2) \leq n^{-C}$ for any $C>0$.
        
        Now assume that $1+\psi \leq \gamma(n) \leq 100$. By
        Corollary~\ref{lem:mindeg_concentration}, there exist $\alpha_1(n)$,
        $\alpha_2(n)$ and a constant $x_1>0$ such that
        \begin{equation*}
         x_1< \alpha_1(n) <\alpha_2(n)<1\text{,\quad}
         \alpha_2(n)-\alpha_1(n) \leq x_1/7
        \end{equation*}
        and
        \begin{equation*}
             \alpha_1pn\leq \deltain\leq \alpha_2 pn
             \quad\text{ and }\quad
   \alpha_1pn\leq \deltaout\leq \alpha_2 pn.
 \end{equation*}

 Let $\zeta >0$ be given by
 Lemma~\ref{lema:expansion_UpperRange_small} applied to
 $\phi=x_1/7$  and $f(n) = \gamma(n)\log n$. Let also $C_2$ be given by 
  Lemma~\ref{lema:expansion_UpperRange_big} for $\zeta$ and $f(n)$.
 
Fix $S\subseteq[n]$. If
 $|S|,|\bar S| \geq \zeta n$, by
 Lemma~\ref{lema:expansion_UpperRange_big}, $d^{in}(S) = |A(\bar
 S,S)| \geq (\zeta^2/2)n^2p \geq 1.5 \alpha_2 n p \geq 1.5\deltain$ with
 probability at least $1-n^{-C_2}$ for sufficiently large $n$. If $2 \leq |S| <
 \zeta n$, by Lemma ~\ref{lema:expansion_UpperRange_small}, with
 probability at least $1-e^{-(\gamma\log n)^2/2}$, for sufficiently
 large $n$,
        \begin{align*}
          d^{in}(S) 	&= \sum_{v \in S}d^{in}(v) - |A[S]|
          \geq \delta^{in}|S| - \phi np|S|
          \\
          &\geq \delta^{in}|S| - \alpha_1 np|S|/7 \geq 6\deltain |S|/7
          \geq 1.5\delta^{in}.
        \end{align*}
        If $2 \leq |\bar S| \leq \zeta n$, then
        \begin{align}
        	d^{out}(\bar S)	&
                =\sum_{v \in \bar S}d^{out}(v) - |A[\bar S]|
        	\geq \delta^{out}|\bar S| - \phi np|\bar S| \\
     		&> |\bar S| n p(\alpha_1 - \phi)  \geq 1.5 np\alpha_2 \geq 1.5\delta^{in}
        \end{align}
        since $2(\alpha_1-\phi)>1.5(\alpha_1+\phi)\geq 1.5\alpha_2$, and the
        result follows since $d^{in}(S) = d^{out}(\bar S)$. Observe that it suffices to fix $C<C_2$ to get the desired result.
       \end{proof}

The next lemma will be useful when we apply Theorem~\ref{teo:Frank} to subpartitions with a very large class, namely subpartitions where one part
contains a $(1-\eps)$-fraction of the vertex set.
\begin{lemma}\label{lemma:critical_bigSet} There exist positive constants
    $\phi$ and $\psi$ such that the following holds. For any function $g = g(n)$
  such that $0\leq g(n) = o(\log n)$, there exist
  positive constants $\epsilon >0$ and $C>0$ such that, for
  any $p=p(n)\in [(\log(n)-g(n))/(n-1),(\log(n)+g(n))/(n-1)]$ and for large $n$, with
  probability at least $1-n^{-C}$, there is no partition
  $(X,Y,Z)$ of $[n]$ of the vertex set of $\Dnp$ satisfying the following conditions:
  \begin{itemize}
  \item[(i)] $|X|\geq (1-\eps)n$,
  \item[(ii)] $|Y|\geq |Z|$ or $|Z|\leq \phi p (n-1)$,
  \item[(iii)] $\din(X)+\din(Y) \leq \psi p (n-1) |Y|$.
  \end{itemize}
\end{lemma}

\begin{proof}
  Let $\phi \leq 3/400$ and $\psi < 3/20$. Let $\zeta > 0$ be
  obtained by Lemma~\ref{lema:expansion_UpperRange_small} with $f(n) =
  \log n -g(n)$ and $\phi$. Let $\epsilon \leq \zeta$ and let
  $(X,Y,Z)$ be a partition of $[n]$ satisfying conditions (i) and (ii) of the
  lemma. By Lemma~\ref{lemma:minDegreeSparse}\ref{eq:smalldindout}
  with $\alpha = 0.09$, we have, for $n$ suficiently
  large, with probability at least $1-n^{-0.2}$, that $d^{in}(v) +
  d^{out}(v) \geq 0.18(n-1)p$ for all $v \in Y$. 
  Lemma~\ref{lema:expansion_UpperRange_small}, we have that $|A[Y]|
  \leq \phi(n-1)p|Y|$ and $| A[Y\cup Z]| \leq \phi(n-1)p(|Y|+|Z|)$
  with probability at least $1-e^{-\log^2n/4}$. Thus,
  \begin{align*}
    d^{in}(X) + d^{in}(Y)
    & = \sum_{v \in Y}\left(d^{in}(v) + d^{out}(v) \right) - 2|A[Y]| - | A (Y,Z)| + | A(Z,X)|. 
    \\
    &\geq
    0.18(n-1)p|Y|-2\phi(n-1)p|Y|-\min\{| A[Y\cup Z]|,|Y||Z|\}
    \\
    &\geq 0.18(n-1)p|Y|-2\phi(n-1)p|Y|-2\phi (n-1) p |Y|
    \\
    &\geq 0.15(n-1)p|Y|,
\end{align*}
as required.
\end{proof}

In the next two lemmas, we bound $\din(S)$ in the range $p \sim \log
n/(n-1)$.

\begin{lemma}\label{lemma:critical_smallSet} 
  Let $g = g(n)$ be a function such that $0\leq g(n) = o(\log n)$.
  There exist positive constants $\eta >0$ and $C>0$
  with the following properties. For all functions $(\log n - g(n))/(n-1) \leq p =
  p(n) \leq (\log n + g(n))/(n-1)$, with probability at least
  $1-n^{-C}$, there are at least two
  vertices with in-degree zero or there is no $S\subseteq [n]$ with size $2\leq
  |S|\leq \eta n$ such that $d^{in}(S) <
  \max\{\delta^{in}+1,2\deltain\}$.
\end{lemma}

\begin{proof}
  Let $p = p(n)$ be a function as in the statement of the lemma and let $\eps>0$ and $C_1$ be obtained through
  Lemma~\ref{lemma:light_vertex} with $\varphi=1.1$.  By Corollary~\ref{lem:mindeg_sparse_o} we have $\deltain < (\eps/16) \log n$ with probability at least $1-n^{C_0}$ for some constant $C_0$.
  Let $\eta=\zeta>0$ given by
  Lemma~\ref{lema:expansion_UpperRange_small} applied to $f(n)=\log{n}-g(n)$ and $\phi \leq \epsilon/16$. 
  Assume that the random digraph has at most one vertex with in-degree zero and fix $S \subseteq [n]$ with size $2\leq
  |S|\leq \eta n$. By Lemma~\ref{lema:expansion_UpperRange_small}, with probability at least $1-e^{-\log^2n/4}$, $|A[S]| \leq \eps n |S| p/16$.
  Let $S_\ell$ denote the $\eps$-in-light vertices in $S$ and let
  $S_h = S\setminus S_\ell$. 

  First assume that all vertices in $S$ are $\eps$-in-light. By Lemma~\ref{lemma:light_vertex}, with
  probability at least $1-n^{-C_1}$, no pair of $\eps$-in-light
  vertices are adjacent, and thus $d^{in}(S) \geq |S|\delta^{in} \geq 2\delta^{in} \geq
  \deltain+1$ if $\deltain>0$. If $\deltain=0$, because there is a single
  vertex with $\din(v)=0$, we have $\din(S) \geq |S|-1\geq
  1=\max\{2\deltain,\deltain+1\}$. Next suppose that there is at
  least one vertex $u$ in $S$ that is not $\eps$-in-light. Note that
  $\din(u)\geq \deltain + \eps np$, which implies that, for $|S|< \eps
  n p/2$, we have $\din(S) \geq |A(\bar{S}, u)|\geq \din(u) -|S|+1\geq
  \deltain+\eps np /2\geq 2\deltain+1$ for large $n$. So we can
  assume that $|S|\geq \eps n p/2$.

  If $|S_h|\geq |S|/8$, then $\din(S) \geq \sum_{v \in S_h}\din(v) - |A[S]|
  \geq |S|\eps n p/8 -  \eps np|S|/16 \geq \epsilon(n-1)p/8 \geq 2\deltain+1$. So assume that $|S_h|\leq
  |S|/8$. Thus, $|S_\ell|\geq 7|S|/8$. Then $\din(S) \geq \deltain
  |S_\ell| - |S_h|$ since no pair of $\eps$-in-light vertices have a
  common in-neighbour by Lemma~\ref{lemma:light_vertex}. If $\deltain > 0$, then $\din(S) \geq |S_\ell|
  - |S_h| \geq 3\epsilon(n-1)p/8\geq 2\deltain+1$. So assume
  that $\deltain = 0$.  Since there is a single vertex with in-degree
  zero, we have $\din(S) \geq |S_\ell|-1 - |S_h| \geq \epsilon(n-1)p/4 \geq 2\deltain+1$.
\end{proof}

\begin{lemma}\label{lemma:out}
  Let $\phi > 0$ be a constant and $g = g(n)$ be a function such that $0\leq g(n) = o(\log n)$. 
  There exist positive constants $\eta >0$ and $C>0$ such that the following holds for all functions $(\log n -
  g(n))/(n-1) \leq p = p(n) \leq (\log n + g(n))/(n-1)$. With
  probability at least $1-n^{-C}$, there exist two vertices
  with in-degree zero or there is no $S\subseteq [n]$ with size
  $\phi\log n \leq |S|\leq \eta n$ such that $d^{out}(S) <2\deltain+1$.
\end{lemma}
\begin{proof}
  Let $\phi>0$ be constant. Let $p = (\log(n)+h(n))/(n-1)$
  such that $|h(n)| \leq g(n)$ and let $\eps>0$ and $C_1$ be obtained
  through Lemma~\ref{lemma:light_vertex} with $\varphi=1.1$.
  We may assume that $\phi< \eps/16$. By
  Corollary~\ref{lem:mindeg_sparse_o} we have  $\deltain < (\phi/16) \log n$ with probability at least $1-n^{C_0}$ for some constant $C_0$. Let $\eta=\zeta>0$   from
  Lemma~\ref{lema:expansion_UpperRange_small} applied to $\phi$ and
  let $\psi \leq \phi/2$. 
  
  If $h(n) < -\log (\psi \log n)$, at least two vertices have
  in-degree zero \aas{} by Chernoff's inequality (Theorem~\ref{teo:chernoff}), so we consider the case $h(n) \geq -\log (\psi \log n)$.
  
  Assume that at most one vertex has in-degree zero and fix $S \subseteq [n]$ with size $\phi \log{n} \leq
  |S|\leq \eta n$. By Lemma~\ref{lema:expansion_UpperRange_small}, $|A[S]| \leq \eps n |S| p/16$ with probability at least $1-e^{-\log^2n/4}$.
 Let $S_\ell$ denote the $\eps$-out-light vertices in $S$ and let
  $S_h = S\setminus S_\ell$. If $|S_h|\geq |S|/8$, then $\dout(S) \geq \sum_{v\in S_h}\dout(v) - |A[S]| \geq
  \eps n p |S|/8 - \epsilon np|S|/16\geq 2\deltain+1$. So
  assume that $|S_h|\leq |S|/8$, and thus $|S_\ell|\geq 7|S|/8$. Then
  $\dout(S) \geq \deltaout |S_\ell| - |S_h|$ since no pair of
  $\eps$-out-light vertices have a common out-neighbour. 
  If $\deltaout
  > 0$, then $\dout(S) \geq |S_\ell| - |S_h| \geq (3/8)\epsilon(n-1)p \geq
  2\deltain+1$.
  We may assume that $\deltaout = 0$.
  We first estimate the number of vertices with out-degree zero.
  The expected number of vertices with out-degree zero is
  $n(1-p)^{n-1} = \exp(-h(n) +o(1))$. It is obvious that the smaller is
  $h(n)$, the higher is the number of vertices with out-degree zero. If
  $h(n) = -\log(\psi\log n)$, then the expected number of
  vertices with out-degree zero is equal to $\psi\log n (1+o(1))$.
  Thus, by Chernoff's inequality (Theorem~\ref{teo:chernoff}), there exists
  $C>0$ such that the probability that there are more than $5\phi \log
  n/8$  vertices with out-degree zero is at most $\exp(-C\log n)$.
  Then, $\dout(S) \geq |S_\ell|-(5\phi/8)\log n-|S|/8\geq 3|S|/4 - 5\phi\log n/8 \geq
  2\deltain+1$ for $h(n) \geq -\log(\psi\log n)$.
\end{proof}

\section{Proof of Theorem~\ref{Th:MainResult}}

In the proof of Theorem~\ref{Th:MainResult}, we consider four different probability regimens, which are described in the following result.
\label{sec:main}
\begin{lemma}\label{lemma:TeoPrinc_Ranges}
  Let $p = p(n) \in[0,1]$. If
  \begin{itemize}
  \item[(i)] $p \leq (\log (n) -g(n))/(n-1)$, for a function $g(n) =
    \Omega(\log{n})$; or
  \item[(ii)] $(\log (n) -g(n))/(n-1) \leq p \leq (\log(n)
    +g(n))/(n-1)$, for a function $g(n) = o(\log{n})$; or
  \item[(iii)] $p \sim (1+\psi)\log(n)/(n-1)$, for a constant $\psi >
    0$; or
  \item[(iv)] $p = g(n)\log(n)/(n-1)$, for a function $g(n)= \omega(1)$,
  \end{itemize}
  then   $\uptau(\cald(n,p))= \lambda(\Dnp)$ a.a.s.
\end{lemma}

We will now show how Theorem~\ref{Th:MainResult} follows from the
lemma above. Let $p = p(n) \in [0,1]$. Let $A_{n}$ be the
event that $\uptau(\cald(n,p)) = \lambda(\Dnp)$ and let $\bar
A_{n}$ be its complement. We
will show that $\lim_{n\to\infty }\Pr(\bar
A_{n}) = 0$. We will use the subsubsequence principle. To this end,
let $(n_i)_{i\in \bbn}$ be an arbitrary increasing sequence
where $n_i\in\bbn$ for all $i\in\bbn$. It suffices
  to show that there is a subsequence $(m_i)_{i\in\bbn}$ of
$(n_i)_{i\in \bbn}$ such that $\lim_{i\to\infty}\Pr(\bar A_{m_i}) =
0$.

Let $h(n) = (n-1)p/\log n$. Note that $h(n)\geq 0$.
Thus, there exists a subsequence $(m_i)_{i\in
    \bbn}$ of $(n_i)_{i\in \bbn}$ such that
$\lim_{i\to\infty}h(m_i)= c$, where $c$ is a nonnegative
constant or $\infty$. If $c<1$, we have
$\lim_{i\to\infty}\Pr(\bar A_{m_i}) = 0$ by
Lemma~\ref{lemma:TeoPrinc_Ranges}(i). If $c = 1$, we have
$\lim_{i\to\infty}\Pr(\bar A_{m_i}) = 0$ by
Lemma~\ref{lemma:TeoPrinc_Ranges}(ii). If $1< c <\infty$, we
have $\lim_{i\to\infty}\Pr(\bar A_{m_i}) = 0$ by
Lemma~\ref{lemma:TeoPrinc_Ranges}(iii). If $c = \infty$, we have
$\lim_{i\to\infty}\Pr(\bar A_{m_i}) = 0$ by
Lemma~\ref{lemma:TeoPrinc_Ranges}(iv). Thus, we may apply the subsubsequence principle and conclude that $\Pr(\bar A_n) = o(1)$.

  \begin{proof}[Proof of Lemma~\ref{lemma:TeoPrinc_Ranges}]
    Case (i): Suppose that $p \leq (\log (n) -g(n))/(n-1)$, for a
    function $g(n) = \omega(1)$. If we have two
    vertices with in-degree $0$, then $\uptau = \lambda = 0$. The
    expected number of vertices with in-degree $0$ is $n(1-p)^{n-1}
    \geq \exp(g(n)+o(1)) \to \infty$. So, by Chernoff's inequality
    (Theorem~\ref{teo:chernoff}), there are at least two vertices with
    in-degree $0$ a.a.s.

    For the remaining cases, let $\calp = (V_0,\dotsc, V_t)$ be a partition of
    $[n]$. We need to show that $\sum_{i=1}^{t} \din(V_i) \geq
    \lambda(t-1)$, since this is equivalent to proving that $\tau(\Dnp) \geq \lambda(\Dnp)$ by Theorem~\ref{teo:Frank}. 

    Case (ii): Fix $g(n)$ given in this case. Let $\phi$ and $\psi$ of Lemma~\ref{lemma:critical_bigSet}, and consider $\eps>0$ and $C>0$ given by it.
    Fix $\eta_1>0$ and $C_1$ as in
    Lemma~\ref{lemma:critical_smallSet} and let $\eta_2>0$ and $C_2$ be obtained by
    applying Lemma~\ref{lemma:out} to $\phi$. Let $\alpha = \min\{\eps,\eta_1,\eta_2\}$. We may assume that the number of vertices with
    in-degree zero is at most one, otherwise
    $\uptau=\lambda=0$ trivially.

    Suppose first that there exists $j>0$ such that $|V_j|\geq
    (1-\alpha)n$. Let $B = \bigcup_{i>0,\ i\neq j} V_i$. Note that the result is trivial if $B = \emptyset$, as $t-1=0$ in this case, thus suppose $|B| >0$.
    If $|V_0|\leq \phi n p$ or $|V_0|\leq |B|$, then $\din(V_j)+\din(B)
    =\Omega(p(n-1)|B|$ with probability $1-n^{-C}$ by Lemma~\ref{lemma:critical_bigSet}.
    Since $\lambda = O(\deltain)$ almost surely by Lemmas~\ref{lemma:lambda_sparsest}
    and~\ref{lem:lambda_critical} and $\deltain = o(\log n)$ by
    Corollary~\ref{lem:mindeg_sparse_o}, we have $\din(V_j)+\din(B)
    =\omega(\lambda|B|)$, so that
    $$\sum_{i=1}^{t} \din(V_i) \geq \din(V_j)+\din(B) \geq \lambda|B| \geq \lambda(t-1).$$
    So assume that $|V_0|> \phi np$ and $|V_0|> |B| \geq 1$. Let $I = \{i>0: |V_i|=1,\ \din(V_i)\leq \lambda \}$. Observe that 
    $\din(V_j) = \dout(V_0 \cup B)$ and by Lemma~\ref{lemma:out} we have $\din(V_j) =
    \dout(V_0 \cup B) \geq 2\deltain+1$, which is at least $\lambda$ \aas\ (see Lemmas~\ref{lemma:lambda_sparsest}
    and~\ref{lem:lambda_critical}).

    Moreover, for every $V_i\in \calp\setminus I$ for $i\neq j$, by
    Lemma~\ref{lemma:critical_smallSet}, $\din(V_i) \geq
    \max\{\deltain+1,2\deltain\}$, which is at least $\lambda$ \aas\
    Let $V(I) = \bigcup_{i\in I} V_i$. Then we have 
    \begin{equation*}
      \sum_{i=1}^t \din(V_i)
      \geq
      \lambda(t-|I|)+\sum_{i\in I} \din(V_i)
      =
      t\lambda  - \sum_{v \in V(I)} (\lambda-\din(v)) 
      \geq
      t\lambda -\lambda,
    \end{equation*}
    by the definition of $\lambda$.

    Next suppose that $|V_i|\leq (1-\alpha)n$ for all $i>0$. If $2\leq
    |V_i|\leq \alpha n$, then $\din(V_i)\geq
    \max\{\deltain+1,2\deltain\}$ by Lemma~\ref{lemma:critical_smallSet}
    and $\din(V_i)\geq \lambda$ \aas\ by Lemmas~\ref{lemma:lambda_sparsest}
    and~\ref{lem:lambda_critical}. If $|V_i|\geq \alpha n$ for some $i$, since we have $|\overline V_i|\geq
    \alpha n$ and so, by Lemma~\ref{lema:expansion_UpperRange_big} with $\zeta=\alpha$,
    $\din(V_i) = |A (\overline V_i,V_i)|   \geq \alpha^2 n^2p/2\geq \lambda$. Thus, \aas
    \begin{equation*}
       \sum_{i=1}^t \din(V_i)
      \geq
      \lambda(t-|I|)+\sum_{i\in I} \din(V_i)
      \geq
      t\lambda -\lambda,
    \end{equation*}
    again by the definition of $\lambda$.

      \noindent Case (iii): By Lemma~\ref{lema:expansion_minDegree},
      every set $S\subseteq[n]$ of size in $[2,n-2]$ has $\din(S)\geq
      1.5\deltain$ \aas{} By Lemma~\ref{lemma:LambdaUpcriticalConst} we have
      $\lambda\sim\deltain$ \aas, which implies that $\din(S)\geq \lambda$. Let
      $I = \{i>0\colon |V_i|=1,\ \din(V_i)\leq\lambda-1 \}$. If $|V_i|\leq n-2$ for all $i$, then \aas\
      $\din(V_i)\geq \lambda$ for all $i\not\in I$ and so
  \begin{equation*}
    \sum_{i=1}^t \din(V_i)
    \geq
    \lambda(t-|I|)+\sum_{i\in I} \din(V_i)
      \geq
      t\lambda -\lambda,
    \end{equation*}
    by the definition of $\lambda$.
    
    Now suppose, without loss of generality, that $|V_1|=n-1$ (the
    case $|V_1|=n$ is trivial). Then there is a single vertex
    $v\not\in V_1$. We can assume that $t = 2$ since the case $t=1$ is
    trivial. Thus, $V_2 = \{v\}$ and $\din(V_1)+\din(V_2) =
    \din(v)+\dout(v)\geq \deltain +\deltaout$. By
    Corollary~\ref{lem:mindeg_concentration} we have $\deltain
    +\deltaout \geq 1.5\deltain$ \aas{} By
    Lemma~\ref{lemma:LambdaUpcriticalConst}, we conclude that
    $\din(V_1)+\din(V_2) \geq 1.5\deltain \geq \lambda$ \aas

    \noindent Case (iv):  We may proceed as in the previous case, since, for every $S\subseteq[n]$ of size in $[2,n-2]$, we again have $\deltain \sim \lambda$ and $\din(S)\geq
    1.5\deltain\geq \lambda$ \aas\ (since $\deltain\sim (n-1)p$ \aas\
    by Lemma~\ref{lemma:concentration_degrees} and $\lambda \sim
    (n-1)p$ \aas\ by Corollary~\ref{corollary:lambda_dense}). This leads to the desired result with the above arguments if $|V_i| \leq n-2$ for every $i$. Otherwise, we use Lemma~\ref{lemma:concentration_degrees} to show that $\deltain+\deltaout \geq 1.5 \deltain$ \aas, and we may again repeat the analysis of the previous case.
%
\end{proof}

\bibliographystyle{amsplain}
\bibliography{bibliografia}
\end{document}